\documentclass[12pt]{amsart}
\usepackage{amsmath,amssymb,amsthm}
\usepackage{color}    
\usepackage{hyperref}
\usepackage{graphicx}
\usepackage{wrapfig} 
 
\usepackage{soul}
\usepackage{ulem}

\usepackage[usenames,dvipsnames]{pstricks}
 \usepackage{epsfig}
 \usepackage{pst-grad} 
 \usepackage{pst-plot} 

\newtheorem{lem}{Lemma}[section]
\newtheorem{prop}{Proposition}[section]
\newtheorem{cor}{Corollary}[section]

\newtheorem{thm}{Theorem}[section]

\theoremstyle{definition}
\newtheorem{definition}{Definition}[section]

\theoremstyle{remark}

\theoremstyle{remark}
\newtheorem{remark}{Remark}[section]
\numberwithin{equation}{section}

\newcommand{\N}{{\mathbb N}}

\newcommand{\R}{{\mathbb R}}

\definecolor{blu}{rgb}{0,0,1}

\title[Uniqueness of solutions]{Remarks on the uniqueness for quasilinear elliptic equations with quadratic growth conditions}

\author[D. Arcoya]{David Arcoya}
\thanks{D. A. is supported by FEDER-MEC (Spain) MTM2012-31799 and Junta de Andaluc\'{\i}a FQM-116.}
\address{David Arcoya 
\newline\indent
Departamento de An\'alisis Matem\'atico, Universidad de Granada, 
\newline\indent
C/Severo Ochoa, 18071 Granada, Spain }
\email{darcoya@ugr.es}

\author[C. De Coster]{Colette De Coster }
\address{Colette de Coster
\newline\indent
Universit\'e de Valenciennes et du Hainaut Cambr\'esis
\newline\indent
LAMAV,  FR CNRS 2956, 
\newline\indent
Institut des Sciences et Techniques de Valenciennes
\newline\indent
F-59313  Valenciennes Cedex 9, France}
\email{Colette.DeCoster@univ-valenciennes.fr}

\author[L. Jeanjean]{Louis Jeanjean}
\address{Louis Jeanjean
\newline\indent
Laboratoire de Math\'ematiques (UMR 6623)
\newline\indent
Universit\'{e} de Franche-Comt\'{e}
\newline\indent
16, Route de Gray 25030 Besan\c{c}on Cedex, France}
\email{louis.jeanjean@univ-fcomte.fr}

\author[K. Tanaka]{Kazunaga Tanaka}
\address{Kazunaga Tanaka
\newline\indent
Department of Mathematics, 
\newline\indent
School of Science and Engineering
\newline\indent
Waseda University
\newline\indent
3-4-1 Ohkubo, Shijuku-ku, Tokyo 169-8555, Japan}
\email{kazunaga@waseda.jp}

\begin{document}


\begin{abstract} In this note we present some uniqueness and comparison results for a class of problem of the form
\begin{equation}
\label{EE0}
\begin{array}{c}
- L u = H(x,u,\nabla u)+ h(x), \quad  
u \in H^1_0(\Omega) \cap L^{\infty}(\Omega),
\end{array}
\end{equation}
where $\Omega \subset \R^N$, $N \geq 2$ is a bounded domain, $L$ is a general  elliptic second order linear operator with bounded coefficients and  $H$ is allowed to have a critical growth in the gradient. In some cases our assumptions prove to be sharp.

\end{abstract}
\maketitle



\section{Introduction}

For  a bounded domain $\Omega \subset \R^N$ ($N \geq 2$) and a function 
$h \in L^p(\Omega)$ for some 
$p > \frac{N}{2}$ we consider the problem
\begin{equation}
\label{EE1}
\begin{array}{c}
- L u = H(x,u,\nabla u)+ h(x), \quad  
u \in H^1_0(\Omega) \cap L^{\infty}(\Omega),
\end{array}
\end{equation}
where  $L$ is a general  elliptic second order linear operator  and  $H: \Omega \times \R \times \R^N \to \R $ is a Carath\'eodory function which satisfy the assumptions:
\vspace{2mm}
\begin{itemize}
\item[(L)] There exists  a family of functions $(a^{ij})_{1\leq i,j\leq N}$ with $a^{ij}\in L^{\infty}(\Omega)\cap W^{1,\infty}_{loc}(\Omega)$ 
  such that
$$
Lu=\sum_{i,j} \frac{\partial}{\partial x_j} \left(a^{ij}(x)\frac{\partial u}{\partial x_i}\right)
$$
and, there exists $\eta>0$ such that, for  a.e.  $x \in \Omega$  and all $ \xi \in \R^N$, 
$$ 
\sum_{i,j}a^{i,j}(x)\xi_i\xi_j\geq \eta |\xi|^2.
$$
\item[(H1)] There exists a continuous function $C_1 : \R^+ \to \R^+$  and a function $ b_1 \in L^p(\Omega)$  such that, for  a.e.  $x \in \Omega,$  all $ u \in \R$ and all $ \xi \in \R^N,$
$$ |H(x,u,\xi)| \leq C_1(|u|)(|\xi|^2 + b_1(x)).$$
\item[(H2)] There exists a function $b_2 \in L_{loc}^N(\Omega)$ and a continuous function $C_2 : \R^+ \times \R^+ \to \R$ such that, for a.e. $ x \in \Omega,$  all $ u_1, u_2 \in \R$  with $ u_1 \geq u_2$ and all $ \xi_1, \xi_2 \in \R^N,$
$$H(x,u_1,\xi_1) - H(x,u_2,\xi_2) \leq C_2(|u_1|,|u_2|) (|\xi_1| + |\xi_2| + b_2(x)) |\xi_1 - \xi_2|.$$  
\end{itemize}
As we shall see in the proof of Corollary~\ref{corolario},   a sufficient condition for (H2)
is that 
for a.e. $x\in \Omega$,
 $H(x,\cdot,\cdot)\in {\mathcal C}^1(\R\times\R^{N})$  with
\begin{equation} \label{H3-1}
\frac{\partial H}{\partial u}(x,u,\xi)\leq 0, \ \mbox{a.e. } x\in \Omega, \ \forall  u\in \mathbb R, \ \forall \xi \in \R^N,
\end{equation}
and that there exists  a function $b_3\in L^N_{loc}(\Omega)$ and a continuous nondecreasing function $C:\mathbb R^+\times \mathbb R^+\to \mathbb R^+$ satisfying
\begin{equation} \label{H3-2}
\left|\frac{\partial H}{\partial \xi}(x,u,\xi)\right| \leq C(|u|)(|\xi|+b_3(x)), \ \mbox{a.e. } x\in \Omega, \ \forall  u\in \mathbb R, \ \forall \xi \in \R^N.
\end{equation}
Uniqueness of solution for problem \eqref{EE1} (with $Lu=\Delta u$) has been first studied in the work 
\cite{BaMu} and after improved in \cite{BaBlGeKo} by requiring weaker regularity 
conditions on the data.   The reader can also see  additional uniqueness results in \cite{B-P} for subcritical nonlinear term $H$  (with respect to $\xi$), i.e, when its growth is less than a power $|\xi|^q$ with $q<2$, and in the work \cite{AS} for the case that $H$ has a singularity at $u=0$.
 
Specifically,  in \cite{BaBlGeKo} the uniqueness of solution for every $h$ is proved  when it is assumed condition \eqref{H3-2} and the following strengthening of \eqref{H3-1}:
$$
\frac{\partial H}{\partial u}(x,u,\xi)\leq-d_0< 0, \ \mbox{a.e. } x\in \Omega, \ \forall  u\in \mathbb R, \ \forall \xi \in \R^N. 
$$
However, in the case that it is only assumed the general hypothesis \eqref{H3-1} (together with \eqref{H3-2}), the authors require to the function $h$ to be sufficiently small in an appropriate sense. Furthermore, adapting the arguments of \cite{BaBlGeKo}, the case where \eqref{H3-1}-\eqref{H3-2} hold and $h$ has sign can also be covered.
Nevertheless,  the treatment of the general case \eqref{H3-1}-\eqref{H3-2} with no assumptions on $h$ seems out of reach with the approach of  \cite{BaBlGeKo,BaMu}.

The special case of (\ref{EE1})  given by
\begin{equation}\label{SC}
- \Delta u = d(x) u + \mu(x) |\nabla u|^2 + h(x), \quad u \in H^1_0(\Omega) \cap L^{\infty}(\Omega)
\end{equation}
 is studied in  \cite{ArDeJeTa} by an alternative approach. Indeed, if $d, h \in L^p(\Omega)$ for some $ p > \frac{N}{2}, \mu \in L^{\infty}(\Omega)$, then it is proved that 
 (\ref{SC}) has at most one  solution as soon as $d \leq 0$. Actually this condition is also necessary since  \cite[Theorem 1.3]{ArDeJeTa}  proves that  (\ref{SC}) may have two solutions if $d\gneqq 0$. See also, in that direction, Theorem 2  in \cite{JeSi} or Theorem 1,(iv)  in \cite{Si}. 
 We also mention that a general condition which guarantees the existence of one solution to (\ref{SC}) is derived in \cite{ArDeJeTa}. This condition always hold when $d <0$ but also widely when $d \leq 0$. For example, we have existence of one solution whenever $\mu$ and $h$ have opposite  sign.
\medbreak

 The aim of this paper is to show that the approach of \cite{ArDeJeTa} can be 
 generalized to treat, under the assumptions (L)-(H1)-(H2), equation (\ref{EE1})  and 
 thus to cover additional situations where the approach of \cite{BaBlGeKo,BaMu} is 
 not applicable.  As a counterpart of our approach we need to assume that the 
 boundary of $\Omega$ is sufficiently smooth, namely that
 $\Omega$ satisfies the following condition (A) of \cite[p.6]{LU68}. 

\begin{definition}
Let $\Omega \subset \R^N $ be an open set.  We say that  $\Omega$ satisfies condition (A)   provided   there exist 
$r_0, \theta_0>0$ such that  if $x\in \partial\Omega$ and
$ 0<r<r_0$, then
$$
         \mbox{meas\,} \Omega_r \leq (1-\theta_0)  \mbox{meas\,}  B_r(x),
$$
for every component $\Omega_r $ of $\Omega\cap B_r(x)$, where $B_r(x)$ denotes the ball of radius $r$ centered at the point $x$.
\end{definition}

 For the reader convenience, we shall prove in the Appendix (see Lemma \ref{3.1}), that condition (A) holds true if $\partial \Omega$ is  Lipschitz.\medbreak

The main result of this paper is the following theorem.

\begin{thm}\label{main} 
Assume that (L)-(H1)-(H2) hold and that  $\Omega$ satisfies condition (A). Then (\ref{EE1}) has at most one solution. 
\end{thm}

The proof of Theorem \ref{main} is divided into two main steps. First we show that any solution of (\ref{EE1}) belongs to ${\mathcal C}^{0,\alpha}(\overline{\Omega}) \cap W^{1,N}_{loc}(\Omega)$ for some $\alpha \in (0,1)$. This result is obtained combining classical regularity arguments from \cite{Gi,LU68} which allow to conclude that it belongs to  ${\mathcal C}^{0,\alpha}(\overline{\Omega}) \cap W^{1,q}_{loc}(\Omega)$ for some $\alpha \in(0,1)$ and some  $q >2$. Here the fact that our solutions are bounded seems essential. Then a bootstrap argument  of \cite{BeFr, Fr} (see also \cite{DoGi}) comes into play. The key ingredient of this bootstrap is an interpolation result due to Miranda \cite {Mi} which says that any element of ${\mathcal C}^{0,\alpha}(\overline{\Omega}) \cap W^{2,\frac{q}{2}}_{loc}(\Omega)$ belongs to $W^{1,t}_{loc}(\Omega)$ for a $t >q$. Having obtained the announced required regularity the second step consists in 
%
%
 establishing a comparison principle. Roughly speaking, we adapt an argument from \cite{BeMeMuPo}, based in turn on an original idea from \cite{BoMa} to show  in Lemma~\ref{UUniqueness} that if $u_1, u_2 \in H^1_0(\Omega) \cap W^{1,N}_{loc}(\Omega) \cap {\mathcal C}(\overline{\Omega})$
are respectively a lower solution and an upper solution of (\ref{EE1}), then $u_1 \leq u_2$ in $\Omega $. The uniqueness then follows from this comparison principle. 
\medbreak

As we already mentioned  Theorem \ref{main} is sharp for equation (\ref{SC}). More globally the condition (H2), which in essence express the fact that $u \to H(x,u,\xi)$ is a non decreasing function and $\xi\to H(x,u,\xi)$ is locally Lipschitz, appears to us as an  almost necessary condition to guarantee the uniqueness.
\vspace{2mm}

Throughout the rest of the note we assume that $N \geq 3$.  The easier case $N=2$ is left to the reader.


\section{Uniqueness results}
\label{Sectionuniqueness-0}

First we show that, when condition (A) holds, any solution of (\ref{EE1}) belongs to ${\mathcal C}^{0, \alpha}(\overline{\Omega})\cap W^{1,q}_{loc}(\Omega)$ 
for some $\alpha \in (0,1)$ and  $q >2$. This follows directly from the following two classical regularity results.

\begin{prop}
\label{Lem1}
Assume that  $\Omega$ satisfies the regularity condition (A) and that condition  (L) holds. Let $u$ be a solution of 
$$ 
-L u + a(x,u, \nabla u) =0, \quad u \in H^1_0(\Omega) \cap L^{\infty}(\Omega).
$$
If there exists a constant $\mu >0$ and a function $b_1 \in L^p(\Omega)$ with $p > \frac{N}{2}$ such that
$$ 
|a(x,u,\xi)| \leq \mu  \left[ |\xi|^2 +b_1(x) \right], \mbox{ a.e. } x\in\Omega, \ \forall u\in \mathbb R, \ \forall \xi \in \mathbb R^N,
$$
then $u\in {\mathcal C}^{0,\alpha}(\overline \Omega)$ for some $\alpha   \in (0,1)$.
\end{prop}

\begin{proof}
This result is a special case of \cite[Theorem IX-2.2 p.441]{LU68}. 
\end{proof}

\begin{prop}
\label{Lem2}
Assume that $L$ satisfies condition (L).
Let $u$ be a solution of 
$$ 
- L u + a(x,u, \nabla u) =0, \quad u \in H^1_0(\Omega) \cap L^{\infty}(\Omega).
$$
If there exists a $\mu >0$ and a function $b_1 \in L^s(\Omega)$ for some $s >1$ such that
$$ 
|a(x,u,\xi)| \leq \mu  \left[ |\xi|^2 + b_1(x)\right], \mbox{ a.e. } x\in\Omega, \ \forall u\in \mathbb R, \ \forall \xi \in \mathbb R^N,
$$
then there exists an exponent $q >2$ such that $u\in W^{1,q}_{loc}( \Omega)$.
\end{prop}

\begin{proof}
This result is a special case of \cite[Proposition 2.1, p.145]{Gi}.
\end{proof} 
\vspace{1mm}

Clearly Propositions \ref{Lem1} and \ref{Lem2} apply to the solutions of (\ref{EE1}).
\vspace{3mm}

The information that an arbitrary solution of (\ref{EE1}) belongs to  ${\mathcal C}^{0, \alpha}(\overline{\Omega})\cap W^{1,q}_{loc}(\Omega)$ for some $\alpha \in (0,1)$ and  $q >2$ is the starting point of a bootstrap argument which relies on the following interpolation result due to C. Miranda \cite{Mi}.

\begin{prop}\label{Lem3} 
Let $ \omega \subset \R^N$ be a bounded domain   satisfying the cone property. Assume that $0 \leq \alpha <1$, $p \geq 1$ and let
$$t = \frac{p(2- \alpha) - \alpha}{1 - \alpha}.$$
Then any element of ${\mathcal C}^{0, \alpha}(\omega) \cap W^{2,p}(\omega)$ belongs to $W^{1,t}(\omega).$
\end{prop}

\begin{proof} This result is \cite[Teorema IV]{Mi}.
\end{proof}

Gathering Propositions \ref{Lem1}, \ref{Lem2} and \ref{Lem3} we obtain

\begin{lem}\label{regu}
Assume that conditions (L) and (H1)  hold and that $\Omega$  satisfies condition (A). Then any solution of (\ref{EE1}) belongs to 
${\mathcal C}^{0,\alpha}(\overline{\Omega}) \cap W^{1,N}_{loc}(\Omega)$ for some $\alpha \in (0,1)$.
\end{lem}

\begin{proof}
Let $u$ be an arbitrary solution of (\ref{EE1}). As $u$ is bounded, we are in position to apply Propositions \ref{Lem1} and \ref{Lem2} with $\mu=\max_{[-\|u\|_{\infty}, \|u\|_{\infty}]}C_1(|u|)$. Then Propositions \ref{Lem1} and \ref{Lem2} implies that $ u \in {\mathcal C}^{0, \alpha}(\overline{\Omega})\cap W^{1,q}_{loc}(\Omega)$ for some $\alpha \in (0,1)$ and  $q >2$. If $q\geq N$ then the proof is done, while if $q<N$ we follow a bootstrap argument of \cite{BeFr, Fr}, see also \cite{DoGi}.  Since $u \in W^{1,q}_{loc}(\Omega)$,  the function $u$ is a weak solution of  
\begin{equation}\label{t1}
 - L u = \xi(x) ,
\end{equation}
with $\xi(x) =   H(x,u,\nabla u)  + h(x) \in L^{\frac{q}{2}}_{loc}(\Omega)$,   by $(H1)$.
By a standard $L^p$-regularity argument, see for 
example  \cite[Theorem~3.8]{Tr87}, we deduce that $u \in W^{2,\frac{q}{2}}_{loc}(\Omega)$. 
Now using Proposition \ref{Lem3} which is valid on any regular domain 
$\omega \subset \Omega$ it follows that 
$$u \in   W^{1,t_1}_{loc}(\Omega)  \quad \mbox{where} \quad t_1 =  \frac{\frac{q}{2} (2- \alpha) - \alpha}{1 - \alpha} >q .$$ If $t_1 \geq N$ we are again done. Otherwise from (\ref{t1}) and    the cited classical regularity argument, $u \in W^{2, \frac{t_1}{2}}_{loc}(\Omega)$. Denoting
\begin{equation}\label{t2}
t_n = \frac{\frac{t_{n-1}}{2} (2- \alpha) - \alpha}{1 - \alpha} > t_{n -1} > q > 2 
\end{equation}
by a bootstrap argument we get $u \in W^{2, \frac{t_n}{2}}_{loc}(\Omega)$ for all $n \in \N$ as long as $t_{n-1} < N$. We now claim that the  increasing  
sequence $\{t_n\}$  exceeds the value  $N$. Indeed, arguing by contradiction, if $t_n<N$ for every $n\in\mathbb N$, then the limit
$l$ of $\{ t_n\}$ has to be 
$l=2$. 
This  contradicts   that $t_n >q >2$. At this point the proof of the lemma is completed.
\end{proof}

The motivation to  observe that any solution of (\ref{EE1}) has an additional regularity appears in the next comparison principle in $H^1(\Omega)\cap W^{1,N}_{loc}(\Omega)\cap {\mathcal C}(\overline\Omega)$. Recall that $u_1$ is a {\it lower solution of \eqref{EE1}} if $u_1^+\in H^1_0(\Omega)$ and, for all $\varphi\in H^1_0(\Omega)\cap L^{\infty}(\Omega)$ with $\varphi\geq0$ a.e. $x\in \Omega$, we have
$$
 \sum_{i,j=1}^N \int_{\Omega} a^{ij} \frac{\partial u_1}{\partial x_i} \frac{\partial \varphi}{\partial x_j}  \leq \int_{\Omega} H(x,u_1,\nabla u_1)\,\varphi + \int_{\Omega} h\,\varphi.
$$
In the same way, $u_2$ is an {\it upper solution of \eqref{EE1}} if $u_2^-\in H^1_0(\Omega)$ and, for all $\varphi\in H^1_0(\Omega)\cap L^{\infty}(\Omega)$ with $\varphi\geq0$ a.e. $x\in \Omega$, we have
$$
\sum_{i,j=1}^N \int_{\Omega} a^{ij} \frac{\partial u_2}{\partial x_i} \frac{\partial \varphi}{\partial x_j}  \geq \int_{\Omega} H(x,u_2,\nabla u_2)\,\varphi + \int_{\Omega} h\,\varphi.
$$

\begin{lem}\label{UUniqueness}
Assume that the hypotheses (L) and (H2) hold. Then if  $u_1, u_2\in H^1(\Omega)\cap W^{1,N}_{loc}(\Omega)\cap {\mathcal C}(\overline\Omega)$ are respectively  a lower solution and an upper solution of \eqref{EE1}, then $u_1\leq u_2$ in $\Omega$. 
\end{lem}


\begin{proof} 
Here we adapt an argument from \cite{BeMeMuPo}, based in turn on an original idea from \cite{BoMa}.
Consider the function $v=u_1-u_2$, which satisfies 
\begin{equation}
\label{2}
\begin{array}{cl}
-L v \leq H(x,u_1,\nabla u_1)-H(x,u_2,\nabla u_2), &\mbox{in }\Omega,
\\
v\leq 0, &\mbox{on } \partial\Omega,
\\
v\in H^1(\Omega)\cap W^{1,N}_{loc}(\Omega) \cap {\mathcal C}(\overline\Omega).&
\end{array}
\end{equation}
\medbreak

For every $c\in\mathbb R$, let us consider the set 
$\Omega_c=\{x\in \Omega \, :\, |v(x)|=c\}$ and 
$$
J=\{c\in \mathbb R  \, :\, \mbox{meas\,}\Omega_c>0\}.
$$
As $|\Omega|$ is finite, $J$ is at most countable and, since for all $c\in \mathbb R$, $\nabla v =0$ a.e. on $\Omega_c$,  we also have
\begin{equation}
				\label{mes}
\nabla v=0 \mbox{ a.e. in } \bigcup_{c\in J} \Omega_c.
\end{equation}
Define
$
Z= \Omega \setminus \bigcup_{c\in J} \Omega_c 
$ and, for all $k\geq 0$, choose 
$ \varphi=(v-k)^+ \in H_0^1(\Omega )\cap L^\infty (\Omega)$ 
 as test function in \eqref{2},
to
 deduce by condition (L) that
$$
\eta \,\|\nabla  (v-k)^+\|^2_{L^2(\Omega)}  \leq 
\int_{\Omega}(H(x,u_1,\nabla u_1)-H(x,u_2,\nabla u_2))\, (v-k)^+ \, dx.
$$
%
Let $A_k=\{ x\in\Omega \ : \ v(x)\geq k\}$.  By \eqref{mes}, (H2) and $\|u_1\|_{\infty}\leq R$, $\|u_2\|_{\infty}\leq R$ for some $R>0$ we obtain a constant $M>0$ such that 
\arraycolsep1.5pt
$$
\begin{array}{rcl}
\displaystyle
\eta \,\|\nabla  (v-k)^+\|^2_{L^2(\Omega)}
&\leq&
\displaystyle
\int_{A_k}(H(x,u_1,\nabla u_1)-H(x,u_2,\nabla u_2))\,  (v-k)^+ \, dx
\\
&\leq&
\displaystyle
M \int_{A_k \cap Z}  (|\nabla u_1|+|\nabla u_2|+b_2)\, | \nabla v |\,  (v-k)^+ \, dx
\\[3mm]
&=&
\displaystyle
M \int_{A_k \cap Z}  (|\nabla u_1|+|\nabla u_2|+b_2)\, | \nabla  (v-k)^+|\,  (v-k)^+ \, dx.
\end{array}
$$
Since  $v \in {\mathcal C}(\overline{\Omega})$ and $v\leq 0$ on $\partial\Omega$,  we have that $(v-k)^+$ has a compact support in $\Omega$, for all  $k >0$,  and hence
$(|\nabla u_1|+|\nabla u_2|+b(x))\in L^N(A_k \cap Z)$. This implies that
\begin{equation}
\label{3bb}
\begin{array}{rcl}
\eta \,\|\nabla  (v-k)^+\|^2_{L^2(\Omega)}
&\leq&
\displaystyle
M \||\nabla u_1|+|\nabla u_2|+b\|_{L^N(A_k \cap Z)}\|\nabla  (v-k)^+\|_{L^2(\Omega)} \| (v-k)^+\|_{L^{2^*}(\Omega)}
\\[2mm]
&\leq&
\displaystyle
\mathcal{S}^{-1}_NM \left\||\nabla u_1|+|\nabla u_2|+b\right\|_{L^N(A_k \cap Z)}\|\nabla (v-k)^+\|^2_{L^2(\Omega)},
\end{array}
\end{equation}
\arraycolsep5pt
where   $2^* =2N/(N-2)$ and $ \mathcal{S}_N$ denotes the Sobolev constant.
\medbreak

We want to prove that $v\leq0$. Indeed, assume by contradiction that $v^+\not\equiv 0$ and consider the non-increasing function  $F$  defined on ${]0, \| v^+\|_\infty]}$ by
$$
F(k)= \mathcal{S}^{-1}_NM \||\nabla u_1|+|\nabla u_2|+b\|_{L^N(A_k \cap Z)}, \ \ \forall \,  0<k<\| v^+\|_\infty
$$
and $F(\| v^+\|_\infty)=0$. By definition of $Z$ we have that $F$ is continuous and we can choose $0<k_0<\| v^+\|_\infty$ such that
$F(k_0) < \eta$. By \eqref{3bb}, 
$\eta \,\|\nabla  (v-k_0)^+\|^2_{L^2(\Omega)}\leq F(k_0) \|\nabla   (v-k_0)^+\|^2_{L^2(\Omega)}$, which implies that $\|\nabla  (v-k_0)^+\|_{L^2(\Omega)}=0$, i.e. 
$v\leq k_0< \| v^+\|_\infty$, a contradiction proving that necessarily $v^+=0$ and hence $u_1\leq u_2$. This concludes the proof.
%
%
%
%
%
\end{proof}


\begin{proof}[Proof of Theorem \ref{main}]
Let $u_1$ and $u_2$ be two solutions of (\ref{EE1}). By Lemma \ref{regu} we know that $u_1$ and $u_2$ belong to ${\mathcal C}(\overline{\Omega}) \cap W^{1,N}_{loc}(\Omega)$. Thus it follows from Lemma \ref{UUniqueness} that $u_1 = u_2$.
\end{proof}

\begin{remark}
In the proof of Theorem \ref{main} the requirement that $\Omega$ satisfies condition (A) is used to show that any solution of (\ref{EE1}) belongs to 
${\mathcal C}(\overline{\Omega})$. In turn this  property is used only in Lemma \ref{UUniqueness} to guarantee that for any solution $u$ of (\ref{EE1}) the set 
$A_k=\{ x \in \Omega : u(x) \geq k\}$ is compact for any $k >0$. It is an open question if the conclusion of Theorem \ref{main} holds true without assumption (A).
\end{remark}

As a corollary of  Lemma~\ref{UUniqueness} and  Theorem \ref{main} we obtain  the following result which, under the condition (A), 
improves 
the results in  \cite{BaBlGeKo,BaMu} concerning (\ref{EE1}).

\begin{cor} \label{corolario}
Assume that $\Omega$ satisfies condition (A) and that condition (L) holds. Let  
$H:\Omega\times\mathbb R\times \mathbb R^N \to \mathbb R$ 
be  a Carath\'eodory function satisfying (H1) and such that, 
for a.e. $x\in \Omega$,
 $H(x,\cdot,\cdot)\in {\mathcal C}^1(\R^{N+1})$ with
$$
\frac{\partial H}{\partial u}(x,u,\xi)\leq 0, \ \mbox{a.e. } x\in \Omega, \ \forall  u\in \mathbb R, \ \forall \xi \in \R^N.
$$
Assume moreover there exists  a function $b_3\in L^N_{loc}(\Omega)$ and a continuous nondecreasing function $C:\mathbb R^+\times \mathbb R^+\to \mathbb R^+$ such that
$$
\left|\frac{\partial H}{\partial \xi}(x,u,\xi)\right| \leq C(|u|)(|\xi|+ b_3(x)), \ \mbox{a.e. } x\in \Omega, \ \forall  u\in \mathbb R, \ \forall \xi \in \R^N.
$$
If $u_1, u_2\in H^1(\Omega)\cap W^{1,N}_{loc}(\Omega)\cap {\mathcal C}(\overline\Omega)$ are respectively  a lower solution and an upper solution of \eqref{EE1}, then $u_1\leq u_2$ in $\Omega$. In particular,  (\ref{EE1}) has at most one  solution.
\end{cor}

\begin{proof}
We just have to prove that (H2) holds. 
Let  $x\in \Omega$,  $u_1$, $u_2\in \mathbb R$ with $u_1\geq u_2$ and $\xi_1$, $\xi_2\in \mathbb R^N$. Define the function $F(t)=H(x,tu_1+(1-t)u_2,t\xi_1+(1-t)\xi_2)$, for every $t\in [0,1]$. Observe that $F\in {\mathcal C}^1(\mathbb R)$. Moreover we have
$$
\arraycolsep1.5pt
\begin{array}{rcl}
H(x,u_1,\xi_1)&-&H(x,u_2,\xi_2)  = F(1)-F(0)
\\
&=&
\displaystyle
\int_0^1 \frac{d}{dt}F(t)\,dt
\\
&=&
\displaystyle
\int_0^1 \frac{\partial H}{\partial u}(x,tu_1+(1-t)u_2,t\xi_1+(1-t)\xi_2)\,dt \, (u_1-u_2)
\\
&&
\displaystyle
+
\int_0^1 \frac{\partial H}{\partial \xi}(x,tu_1+(1-t)u_2,t\xi_1+(1-t)\xi_2) \cdot (\xi_1-\xi_2)\,dt
\\
&\leq& 
\displaystyle
\int_0^1 \left|\frac{\partial H}{\partial \xi}(x,tu_1+(1-t)u_2,t\xi_1+(1-t)\xi_2)\right|\,dt\,\, |\xi_1-\xi_2|
\\
&\leq& 
\displaystyle
\int_0^1 C_0(|tu_1+(1-t)u_2|)\, \left[|t\xi_1+(1-t)\xi_2|+ b_3(x)  \right]\,dt\,\, |\xi_1-\xi_2|
\\
&\leq& C_0(|u_1|+|u_2|)\, \left[|\xi_1|+|\xi_2|+ b_3(x)  \right]\,|\xi_1-\xi_2|.
\end{array}
$$
\arraycolsep5pt 
This proves that (H2) is valid and we can apply  Lemma~\ref{UUniqueness} and  Theorem \ref{main} to conclude.\end{proof}


\appendix
\section{Sufficient conditions for  condition (A)}

We prove in this section that  $\Omega$ satisfies condition (A)  whenever $\partial \Omega$ is  Lipschitz. 
Recall   that, by \cite[Theorem 1.2.2.2]{Gr},   $\partial \Omega$ is  Lipschitz if and only if  the uniform  cone condition is satisfied.

\begin{definition} (\cite[Definition 1.2.2.1, p.10]{Gr})
Let $\Omega$ be an open subset of $\mathbb R^N$. We say that $\Omega$ satisfies  the {\it uniform cone property}
if, for every $x\in \partial\Omega$, there exists a neighbourhood $V$ of $x$ in $\mathbb R^N$ and new coordinates $\{y_1,\ldots, y_N\}$ such that

\noindent
\hangindent7mm%
\makebox[7mm][l]{(a)}%
$V$ is a hypercube in the new coordinates, i.e.,
$$
V=\{(y_1,\ldots, y_N) \in\mathbb R^N \mid -a_i<y_i<a_i, \,\, 1\leq i \leq N\}, 
$$
for some $a_i>0$, $i=1,2,\dots,N$;

\noindent
\hangindent7mm%
\makebox[7mm][l]{(b)}%
$y-z\in \Omega$ whenever $y\in \overline\Omega \cap V$ and $z\in C$, where $C$ is the open cone $\{z=(z',z_n)\mid (\cot \theta)|z'|<z_n<h\}$ for some $\theta \in \,]0,\pi/2]$ and some $h>0$.
\end{definition}

\begin{lem}\label{3.1}  
Let $\Omega \subset \R^N$ be a bounded domain. If $\partial \Omega$ is Lipschitz,
then $\Omega$ satisfies  the condition (A).
\end{lem}

\begin{proof} 
As it has been mentioned we can assume that  $\partial \Omega$ satisfies 
the  uniform cone condition.
Arguing as in the proof of \cite[Theorem 1.2.2.2]{Gr}, we know that $\{x\}-C\subset \Omega$ but we can also observe that $\{x\}+C\subset \Omega^c$, at least if the distance from $x$ to $V^c$ is greater that $h/\cos\theta$; this last condition can always be achieved by choosing a smaller $h$. Indeed, if $(\{x\}+C)\cap \overline\Omega$ is not empty, let $y$ be a point in the intersection.  Then $y\in \overline\Omega\cap V$ since $|y_n-x_n|<h$ and  consequently, $\{y\}-C\subset \Omega$, but this contradicts the fact that $x\in \{y\}-C$.

For every $r>0$, we have that 
$\Omega \cap B_r(x) \subset B_r(x) \setminus (B_r(x)\cap (\{x\}+C))$ and hence
\begin{eqnarray*}
 \mbox{meas\,} (\Omega \cap B_r(x) ) &\leq  &
  \mbox{meas\,} B_r(x) -
    \mbox{meas\,} (B_r(x)\cap (\{x\}+C)) 
\\
&=&    \mbox{meas\,} B_r(x)  \left( 1 - \frac{ \mbox{meas\,} (B_r(x)\cap (\{x\}+C)) }{ \mbox{meas\,} B_r(x)} \right)
\\
&=&
\mbox{meas\,} (B_r(0)\cap C).
\end{eqnarray*}


If $r\leq h$, the cone $C(\theta , r \cos \theta)$ of
vertex $0$, opening angle
$\theta$ and height $r\cos \theta $ satisfies
$$
C(\theta , r \cos \alpha) \subset  (B_r(0)\cap C).
$$
%
%
%
Hence, 
$$
\frac{1}{N+1}\left(\mbox{meas}_{\mathbb R^{N-1}} B_{\mathbb R^{N-1}}  (r\sin \theta)\right) r\cos \theta \leq \mbox{meas\,}(B_r(0)\cap C),
$$
and
therefore
\arraycolsep1.5pt
\begin{eqnarray*}
 \mbox{meas\,} (\Omega \cap B_r(x) ) &\leq  &     \mbox{meas\,} B_r(x) 
\left( 1 - \frac{ \mbox{meas\,} (B_r(0)\cap C) }{ \mbox{meas\,} B_r(0)} \right)
\\
&\leq &
  \mbox{meas\,} B_r(x) 
  \left( 1 - \frac{\frac{1}{N+1}\frac{\pi^{(N-1)/2}}{\Gamma ( \frac{N-1}{2}+1)}(r\sin \theta)^{N-1} r\cos \theta }{ \frac{\pi^{N/2}}{\Gamma (\frac{N}{2}+1)} r^N} \right)
\\
&=&  \mbox{meas\,} B_r(x) \left(  1 - \frac{\Gamma (\frac{N}{2}+1)\, (\sin \theta)^{N-1}  \cos \theta}{(N+1) \pi^{1/2}\,\Gamma ( \frac{N-1}{2}+1)} 
\right),
\end{eqnarray*}
\arraycolsep5pt
i.e., condition (A) holds with $r_0=h$ and 
$\theta_0=\frac{\Gamma (\frac{N}{2}+1)\, (\sin \theta)^{N-1}  \cos \theta}{(N+1) \pi^{1/2}\,\Gamma ( \frac{N-1}{2}+1)} $. 
\end{proof}

%
%


\end{document}